\documentclass{jaums}

\usepackage{amsmath,amsfonts,amssymb}
\usepackage{verbatim}
\usepackage{enumerate}
\usepackage{amsthm}
\usepackage{url}
\usepackage{multicol}
\usepackage[colorlinks]{hyperref}
\usepackage[english]{babel}

\newcommand{\lcm}{\mathrm{lcm}}

\newcommand{\Tr}{\mathrm{Tr}}
\newcommand{\fqn}{\mathbb{F}_{q^n}}
\newcommand{\F}{\mathbb{F}}

\theoremstyle{thmit} 
\newtheorem{thm}{Theorem}[section]
\newtheorem{lem}[thm]{Lemma}

\theoremstyle{thmrm} 
\newtheorem{defi}[thm]{Definition}

\newtheorem*{rem}{Remark}
\newtheorem*{oldproof}{Proof}
\renewenvironment{proof}[1][{}]{\begin{oldproof}[#1]}{\qed\end{oldproof}}

\title{Generators of finite fields with prescribed traces}

\author[Lucas Reis]{Lucas Reis}
\address{Departamento de Matem\'{a}tica, Universidade Federal de Minas Gerais (UFMG),\\
Belo Horizonte, MG, 30270-901, Brazil.}
\email{lucasreismat@ufmg.br}

\author[S\'avio Ribas]{S\'avio Ribas}
\address{Departamento de Matem\'{a}tica, Universidade Federal de Ouro Preto (UFOP),\\
Ouro Preto, MG, 35400-000, Brazil.}
\email{savio.ribas@ufop.edu.br}


\begin{document}

\maketitle

\begin{abstract}
This paper explores the existence and distribution of primitive elements in finite field extensions with prescribed traces in several intermediate field extensions. Our main result provides an inequality-like condition to ensure the existence of such elements. We then derive concrete existence results for a special class of intermediate extensions.
\end{abstract}



\section{Introduction}

Given a prime power $q$ and $n>1$ a positive integer, let $\F_q$ be the finite field with $q$ elements and $\F_{q^n}$ the unique $n$-degree field extension of $\F_q$. The intermediate extensions of $\F_{q^n}$ over $\F_q$ are exactly the finite fields $\F_{q^d}$ with $d$ a divisor of $n$. It is well-known that the multiplicative group $\F_{q^n}^*$ is cyclic; any generator of such group is called {\em primitive}. Primitive elements play important roles in a wide variety of applications in Cryptography and, perhaps, the most notably is the  Diffie-Hellman key exchange~\cite{dh}. Primitive elements with further specified properties have been extensively studied in the past few decades. The motivation comes from both theoretical and practical matters. 

For instance, the celebrated {\em Primitive Normal Basis Theorem} states that, for any $n\ge 1$ and any prime power $q$, there exists a primitive element $\alpha\in \F_{q^n}$ such that $\alpha$ is {\em normal} over $\F_q$, i.e., the set $\{\alpha, \alpha^q, \ldots, \alpha^{q^{n-1}}\}$ comprises an $\F_q$-basis for $\F_{q^n}$. The  Primitive Normal Basis Theorem was proved by Lenstra and Schoof~\cite{lenstra} and a proof without any use of computers was later given by Cohen and Huczynska~\cite{cohen}. Cohen~\cite{Cohentrace} also explored the existence of primitive elements in $\F_{q^n}$ with prescribed {\em trace} $a\in \F_q$, i.e., primitive elements $\alpha\in \F_{q^n}$ such that 
\begin{equation}\label{eq:trace-intro}
\Tr_{q^n/q}(\alpha)=\sum_{i=0}^{n-1}\alpha^{q^i}=a.
\end{equation} 
He has shown that, up to genuine exceptions, it is possible to find primitive elements in $\F_{q^n}$ satisfying equation \eqref{eq:trace-intro}. More specifically, we have the following result.

\begin{thm}\label{thm:cohen}
Let $q$ be a prime power, $n$ a positive integer and $a\in \F_q$. Then there exists a primitive element $\alpha\in \F_{q^n}$ such that $\Tr_{q^n/q}(\alpha)=a$ unless $a=0$ and $n=2$ or $a=0$, $n=3$ and $q=4$.
\end{thm}

In this paper, we discuss the existence of primitive elements in $\F_{q^n}$ with prescribed traces in several intermediate extensions $\F_{q^d}$ of $\F_{q^n}$. In other words, given $n>1$, $d_1<\ldots<d_k<n$ divisors of $n$ and $a_j\in \F_{q^{d_i}}$, we discuss the existence of a primitive element $\alpha \in \F_{q^n}$ such that, for each $1\le j\le k$, we have that
$$\Tr_{q^n/q^{d_j}}(\alpha)=\sum_{i=0}^{n/d_j-1}\alpha^{q^{id_j}}=a_j.$$
Our main result, Theorem~\ref{thm:main}, provides an inequality-like condition to ensure the existence of such elements, that might easily yield asymptotic existence results.  We then present a special instance where we can obtain effective results. In particular we prove that, up to few critical cases, there exists a primitive element $\alpha\in \F_{q^n}$ with arbitrary prescribed traces in any two intermediate $\F_q$-extensions of $\F_{q^n}$: see Theorem~\ref{thm:cop} for more details.

The structure of the paper is as follows. In Section 2 we introduce some useful notation and present our main result. Section 3 provides background material that is used along the way and some auxiliary results. In Section 4 we prove our main result. Finally, in Section 5, we restrict our problem to a special class of intermediate extensions, where our results are sharpened.

\section{Main Results}

Before we state our main result, we introduce some notation and discuss on a natural condition that we have to impose in the problem. Throughout this paper, $q$ is a prime power and $\F_q$ is the finite field with $q$ elements. 

\begin{defi}
For $n>1$, $d$ a divisor of $n$ and $\alpha\in \F_{q^n}$, we set
$$\Tr_{n/d}(\alpha)=\sum_{i=0}^{n/d-1}\alpha^{q^{di}},$$
the {\em trace} of $\alpha$ over $\F_{q^d}$.
\end{defi}

Recall that the trace is transitive, i.e., if $e$ divides $d$ and $d$ divides $n$, then for any $\alpha\in \F_{q^n}$ we have that
$\Tr_{n/e}(\alpha)=\Tr_{n/d}(\Tr_{d/e}(\alpha))$.
In particular, if $d_1<\ldots<d_k<n$ are divisors of $n$ and we pick $a_i\in \F_{q^{d_i}}, 1\le i\le k$, then the existence of an element $\alpha\in \F_{q^n}$ with $\Tr_{n/d_i}(\alpha)=a_i$ is necessarily conditional on the following identities
\begin{equation}\label{eq:conditions}\Tr_{d_i/\gcd(d_i, d_j)}(a_i)=\Tr_{n/\gcd(d_i, d_j)}(\alpha)=\Tr_{d_j/\gcd(d_i, d_j)}(a_j),\; 1\le i, j\le k.\end{equation}

\begin{rem}\label{rem:reis}
As recently shown by the first author in~\cite{LR19}, the equations in \eqref{eq:conditions} are also sufficient and, in this case, there exist exactly $q^{n-\lambda}$ elements in $\F_{q^n}$  with $\Tr_{n/d_i}(\alpha)=a_i$ for $1\le i\le k$, where
\begin{align*}\lambda &=\deg(\lcm(x^{d_1}-1, \ldots, x^{d_k}-1))\\ & =d_1+\ldots+d_k+\sum_{i=2}^{k}(-1)^{i+1}\sum_{1\le \ell_1<\cdots<\ell_i\le k}\gcd(d_{\ell_1},\ldots, d_{\ell_i}).\end{align*}
The proof of this result is a simple application of the Chinese Remainder Theorem for the ring $\F_q[x]$. For more details, see Theorem~4.1 in~\cite{LR19}.
\end{rem}

The equations in \eqref{eq:conditions} imply that if $d_i$ divides some $d_j$, then the equality $\Tr_{n/d_i}(\alpha)=a_i$ is already implied by 
$\Tr_{n/d_j}(\alpha)=a_j$. So we may restrict ourselves to divisors $d_1<\ldots<d_k$ of $n$ such that $d_i\nmid d_j$ for any $1\le i< j\le k$. In addition, the case $k=1$ is completely settled by Cohen~\cite{Cohentrace}, so we assume that $k>1$, i.e., $n$ is not a prime power. We introduce a useful notation.

\begin{defi}\label{lambdas}
Let $n>1$ be an integer that is not a prime power and $1<k< \sigma_0(n)$, where $\sigma_0(n)$ denotes the number of positive divisors of $n$.
\begin{enumerate}[(i)]
\item $\Lambda_k(n)$ stands for the set of $k$-tuples $\mathbf{d}=(d_1, \ldots, d_k)$, where $d_1<\ldots<d_k<n$ are divisors of $n$ such that $d_i$ does not divide $d_j$ for every $1\le i, j\le k$ with $i\ne j$. 
\item For $\mathbf{d} = (d_1, \ldots, d_k) \in \Lambda_k(n)$, set $\F(\mathbf{d})=\prod_{i=1}^k\F_{q^{d_i}}$ and  
$$\lambda(\mathbf{d})=d_1+\ldots+d_k+\sum_{i=2}^{k}(-1)^{i+1}\sum_{1\le \ell_1<\cdots<\ell_i\le k}\gcd(d_{\ell_1},\ldots, d_{\ell_i}).$$
\end{enumerate}
Moreover, for $\mathbf{d} = (d_1, \ldots, d_k) \in \Lambda_k(n)$ and $\mathbf{a} = (a_1, \ldots, a_k) \in \F(\mathbf{d})$, the $k$-tuple $\mathbf{a}$ is {\em $\mathbf{d}$-admissible} if, for any $1\le i<j\le k$, we have that
$$\Tr_{d_i/\gcd(d_i, d_j)}(a_i)=\Tr_{d_j/\gcd(d_i, d_j)}(a_j).$$
\end{defi}
From previous observation, we only need to consider $\mathbf{d}$-admissible $k$-tuples. Our main result can be stated as follows.

\begin{thm}\label{thm:main}
Let $n>1$ be an integer that is not a prime power, $1<k<\sigma_0(n)$, $\mathbf{d} = (d_1, \ldots, d_k) \in \Lambda_k(n)$, and $\mathbf{a}=(a_1, \ldots, a_k)\in \F(\mathbf{d})$ a $\mathbf{d}$-admissible $k$-tuple. Then there exists a primitive element $\alpha \in \F_{q^n}$ with prescribed traces $\Tr_{n/d_i}(\alpha)=a_i$ for every $1 \le i \le k$ provided that \begin{equation}\label{ineqprincipal}
q^{n/2-\lambda(\mathbf{d})} \ge W(q^n-1),
\end{equation} where $W(t)$ denotes the number of squarefree divisors of $t$.
\end{thm}

In the context of Theorem~\ref{thm:main}, we also obtain the following minor result.

\begin{thm}\label{thm:main-2}
We have that Theorem~\ref{thm:main} also holds if the condition $q^{n/2-\lambda(\mathbf{d})}\ge W(q^n-1)$ is replaced by the inequality $\lcm(d_1, \ldots, d_k)<n$.
\end{thm}

While the proof of Theorem~\ref{thm:main-2} is a straightforward combination of Theorem~4.1 in~\cite{LR19} and Theorem~\ref{thm:cohen}, the proof of Theorem~\ref{thm:main} relies on character sum methods to count elements in finite fields with specified properties. We follow the traditional approach that is presented in~\cite{Cohentrace, cohen, lenstra}. In this approach, we frequently need to simplify character sums by detecting trivial Gauss sums; in our case, we employ a result from~\cite{LR19} concerning special zero sums in finite fields.

\section{Preliminaries}

This section provides background material that is used throughout the paper and some auxiliary results.

\subsection{Characters and characteristic functions}

Here we provide character sum expressions for the characteristic functions of elements in finite fields with specified properties. We start by recalling some basics on characters over finite fields.

Fix $\alpha\in \F_{q^n}$ a primitive element. A typical multiplicative character of $\F_{q^n}$ is a function $\eta:\F_{q^n}^*\to \mathbb C$ given by $\eta(\alpha^k)=e^{\frac{2\pi i k t}{q^n-1}}$ for some positive integer $t\le q^n-1$. The character $\eta_1\equiv 1$ is the trivial multiplicative character.  The set of multiplicative characters of $\F_{q^n}$ form a (multiplicative) cyclic group of order $q^n-1$. In particular, for each divisor $t$ of $q^n-1$, there exist exactly $\varphi(t)$ multiplicative characters of order $t$; we denote the set of such characters by $\Gamma(t)$. We extend the evaluation of multiplicative characters to the element $0\in \F_{q^n}$ by letting $\eta(0)=0$.

If $p$ is the characteristic of $\F_q$, say $q=p^s$, and $m$ is any divisor of $n$, the canonical additive character of $\F_{q^m}$ is the function $\chi:\F_{q^m}\to \mathbb C$ given by $$\chi(\beta)=e^{\frac{2\pi i \mathcal T_{m}(\beta)}{p}},$$
where $\mathcal T_{m}(\beta)=\sum_{i=1}^{ms-1}\beta^{p^i}\in \F_p$ is the absolute trace function from $\F_{q^m}$ to $\F_p$. For each $c\in \F_{q^m}$, we set $\chi_{c}(\beta)=\chi(c\cdot \beta)$, which is another additive character of $\F_{q^m}$. In fact, the set of additive characters of $\F_{q^m}$ is a (multiplicative) group isomorphic to the  additive group $\F_{q^m}$ and comprise the characters $\{\chi_{c}\,|\, c\in \F_{q^m}\}$. The identity of such group is the trivial additive character $\chi_0$. We introduce a useful notation.

\begin{defi}
Fix $n$ a positive integer, $d$ a divisor of $n$ and $a\in \F_{q^d}$. Let $I_{n, d, a}$ be the characteristic function for elements in $\F_{q^n}$ with trace $a$ over $\F_{q^d}$, and let $\Omega_{n}$ be the characteristic function for primitive elements in $\F_{q^n}$, i.e., for $\alpha\in \F_{q^n}$ we have that
$$I_{n, d, a}(\alpha)=\begin{cases}1 & \text{if}\;\; \Tr_{n/d}(\alpha)=a,\\ 0 & \text{otherwise,}\end{cases} \quad \text{and} \quad \Omega_{n}(\alpha)=\begin{cases}1 & \text{if $\alpha$ is primitive,} \\ 0 & \text{otherwise.}\end{cases}$$
In addition, set $\theta(q)=\frac{\varphi(q^n-1)}{q^n-1}$.
\end{defi}

The following results provide expressions for the functions $\Omega_n$ and $I_{n, d, a}$ by means of characters.

\begin{lem}\label{aux1}
For every $\beta\in \F_{q^n}$, we have that
$$\Omega_n(\beta)=\theta(q)\sum_{t|q^n-1}\frac{\mu(t)}{\varphi(t)}\sum_{\eta \in \Gamma(t)}\eta(\beta),$$
where $\mu$ is the Möbius function over the integers.
\end{lem}

For the proof of the previous lemma, see Theorem~2.8 of~\cite{KR} and the comments thereafter.

\begin{lem}\label{aux2}
Let $m$ be a divisor of $n$ and $\gamma\in \F_{q^n}$ such that $\Tr_{n/m}(\gamma)=a\in \F_{q^m}$. If $\chi$ denotes the  canonical additive character of $\F_{q^n}$, then for any $\beta\in \F_{q^n}$ we have that  
$$I_{n, d, a}(\beta)=\frac{1}{q^m}\sum_{c\in \F_{q^m}}\chi_{c}(\beta-\gamma)=\frac{1}{q^m}\sum_{c\in \F_{q^m}}\chi_{c}(\beta)\chi_c(\gamma)^{-1}.$$
\end{lem}

For the proof of the previous lemma, see Subsection 2.3.1 of~\cite{KR}.

\subsection{Auxiliary lemmas}

From Corollary 1.2 in~\cite{LR19}, we have the following result.

\begin{lem}\label{lem:zero-sum}
Let $n>1$ be an integer that is not a prime power, $1<k<\sigma_0(n)$, and let $\mathbf{d} =(d_1, \ldots, d_k)\in \Lambda_k(n)$. Then the number of $k$-tuples $(x_1, \ldots, x_k)\in \F(\mathbf{d})$ such that $x_1+\cdots+x_k=0$
equals $$q^{d_1+\cdots+d_k-\lambda(\mathbf{d})}.$$
\end{lem}

We further require effective upper bounds on the functions $W$ and $\lambda(\mathbf{d})$. We have the following results.

\begin{lem}\label{lem:tech-1}
\begin{enumerate}[(i)]
\item If $W(t)$ is the number of squarefree divisors of $t$, then for all $t\ge 3$,
$$W(t-1)<t^{\frac{0.96}{\log\log t}}.$$

\item If $n>1$ is an integer that is not a prime power, $1<k<\sigma_0(n)$ and $\mathbf{d}\in \Lambda_k(n)$, then
$$\lambda(\mathbf{d})\le n-\varphi(n),$$
where $\varphi(n)$ is the Euler Totient Function.
\end{enumerate}
\end{lem}

\begin{proof}
Item (i) is a straightforward consequence of inequality (4.1) in~\cite{cohen-trud}. For item (ii) observe that, as stated in Remark~\ref{rem:reis}, $\lambda(\mathbf{d})$ equals the degree of the least common multiple of the polynomials $x^{d_1}-1, \ldots, x^{d_k}-1$. Since $d_1<\cdots<d_k<n$, if $p_1, \ldots, p_t$ are the distinct prime divisors of $n$, we have that the polynomials  $x^{d_1}-1, \ldots, x^{d_k}-1$ divide the polynomial
$$\lcm(x^{n/p_1}-1, \ldots, x^{n/p_t}-1).$$
An inclusion-exclusion argument shows that the previous polynomial has degree $n-\varphi(n)$ and the result follows.
\end{proof}

\begin{lem}[see Lemma 4.1 of~\cite{KR}]\label{lem:tech-2}
If $a$ is a positive integer and $p_1, \dots, p_j$ are the distinct prime divisors of $t$ such that $p_i \le 2^a$, then $$W(t) \le c_{t,a} t^{1/a}, \quad \quad \text{where} \quad c_{t,a} := \frac{2^j}{(p_1 \dots p_j)^{1/a}}.$$
In particular, 
$$c_{t,4} < 
\begin{cases}
4.9 \quad \text{for $t$ even,} \\
2.9 \quad \text{for $t$ odd,}
\end{cases}
\quad \text{and} \quad \quad  
c_{t,8} < 4514.7.$$
\end{lem}

\section{Proof of the main result}

Let $N(n, \mathbf{d}, \mathbf{a})$ be the number of primitive elements $\alpha\in \F_{q^n}$ such that $\Tr_{n/d_i}(\alpha)=a_i$. In particular, 
$$N(n, \mathbf{d}, \mathbf{a})=\sum_{w\in \F_{q^n}}\Omega_{n}(\omega)\cdot \prod_{i=1}^k I_{n, d_i, a_i}(w).$$
Since the $k$-tuple $(a_1, \ldots, a_k)$ is $\mathbf{d}$-admissible, we have seen that there exists $\beta\in \F_{q^n}$ such that $\Tr_{\frac{n}{t}/d_i}(\beta)=a_i$ for $1\le i\le k$. Write $D=d_1+\cdots+d_k$ and, for a generic $\mathbf{c} = (c_1, \ldots, c_k) \in \F(\mathbf{d})$, write $s(\mathbf{c})=\sum_{i=1}^kc_i$.
From Lemmas~\ref{aux1} and~\ref{aux2}, we have that
\begin{align*}
\frac{q^{D}N(n, \mathbf{d}, \mathbf{a})}{\theta(q)} &= \sum_{w\in \F_{q^n}} \sum_{t|q^n-1}\frac{\mu(t)}{\varphi(t)}\sum_{\eta \in \Gamma(t)}\eta (w)\cdot \prod_{i=1}^k\left(\sum_{c_i\in \F_{q^{d_i}}}\chi_{{c_i}}(w)\cdot \chi_{{c_i}}(\beta)^{-1}\right) \\
&= \sum_{w\in \F_{q^n}}\sum_{\mathbf{c}\in \F(\mathbf{d})}\sum_{t|q^n-1}\frac{\mu(t)}{\varphi(t)}\sum_{\eta \in \Gamma(t)}\eta(w)\cdot \chi_{s(\mathbf{c})}(w)\cdot \chi_{s(\mathbf{c})}(-\beta)
\end{align*}
$$= \sum_{\mathbf{c}\in \F(\mathbf{d})}\sum_{t|q^n-1}\frac{\mu(t)}{\varphi(t)}\sum_{\eta \in \Gamma(t)}\chi_{s(\mathbf{c})}(-\beta)\cdot G_{n}(\eta, \chi_{s(\mathbf{c})}),$$
where $G_n(\eta, \chi_{s(\mathbf{c})})=\sum_{w\in \F_{q^n}}\eta(w)\cdot \chi_{s(\mathbf{c})}(w)$ denotes a Gauss sum. 
We use the orthogonality relations to obtain
$$G_{n}(\eta, \chi_{s(\mathbf{c})})=
\begin{cases}
q^n & \text{if}\;\; \eta \in \Gamma(1)\;\;\text{and}\;\; s(\mathbf{c})= 0,\\
0 & \text{if}\;\; \eta \in \Gamma(1)\;\;\text{and}\;\; s(\mathbf{c})\ne 0,\\  
0 & \text{if}\;\; \eta\not\in\Gamma(1) \;\;\text{and}\;\;s(\mathbf{c})= 0,
\end{cases}$$
and in the remaining cases we have the well-known identity $|G_{n}(\eta, \chi_{s(\mathbf{c})})|=q^{n/2}$. In addition, if $s(\mathbf{c})=0$, then $\chi_{s(\mathbf{c})}(-\beta)=\chi_{0}(-\beta)=1$. In particular, we may rewrite

$$\frac{q^{D}N(n, \mathbf{d}, \mathbf{a})}{{\theta(q)}}=\sum_{\mathbf{c}\in \F(\mathbf{d})\atop s(\mathbf{c})=0}q^n+\underbrace{\sum_{\mathbf{c}\in \F(\mathbf{d})\atop s(\mathbf{c})\ne 0}\sum_{t|q^n-1\atop t\ne 1}\frac{\mu(t)}{\varphi(t)}\sum_{\eta \in \Gamma(t)}\chi_{s(\mathbf{c})}(-\beta)\cdot G_{n}(\eta, \chi_{s(\mathbf{c})})}_{S}.$$

From Lemma~\ref{lem:zero-sum}, we obtain the following equality
\begin{equation}
\label{eq:aux-eq}\frac{q^{D}N(n, \mathbf{d}, \mathbf{a})}{{\theta(q)}}=q^{n+D-\lambda(\mathbf{d})}+S.
\end{equation}

We observe that $|\chi_{s(\mathbf{c})}({-\beta})|=1$ and $|G_n(\eta, \chi_{s(\mathbf{c})})|=q^{n/2}$ in every term of the sum $S$. Recall that there exist exactly $\varphi(t)$ elements in $\Gamma(t)$ and the function $\mu$ has absolute value $1$ at squarefree integers, and vanishes everywhere else. In particular, we obtain the following inequality 
$$|S|<\sum_{\mathbf{c}\in \F(\mathbf{d})\atop s(\mathbf{c})\ne 0}\sum_{t|q^n-1\atop t\; \text{squarefree}}q^{n/2}=q^{n/2+D}\cdot W(q^n-1).$$
Hence,
$$\frac{q^{D}N(n, \mathbf{d}, \mathbf{a})}{{\theta(q)}}>q^{n+D-\lambda(\mathbf{d})}-q^{n/2+D}\cdot W(q^n-1)\ge 0,$$
provided that $q^{n/2-\lambda(\mathbf{d})}\ge W(q^n-1)$.

\subsection{Proof of Theorem~\ref{thm:main-2}}

Set $\lcm(d_1, \ldots, d_k)=\frac{n}{t}$, where $t>1$ is a divisor of $n$. Since there do not exist $1\le i, j \le k$ such that $d_i$ divides $d_j$, we have that $d_i<n/t$ for any $1\le i\le k$ and then, from Lemma~\ref{lem:tech-1}, we have that $\lambda(\mathbf{d})\le n/t-\varphi(n/t)$. From hypothesis the $k$-tuple $(a_1, \ldots, a_k)$ is $\mathbf{d}$-admissible and then, as stated in Remark~\ref{rem:reis}, Theorem 4.1 of~\cite{LR19} implies that there exist $q^{n/t-\lambda(\mathbf{d})}\ge q^{\varphi(n/t)}>1$ elements $\theta\in \F_{q^{n/t}}$ such that $\Tr_{\frac{n}{t}/d_i}(\theta)=a_i$ for $1\le i\le k$. In particular, there exists $\theta_0\in \F_{q^t}$ with such a property in a way that $\theta_0\ne 0$.  Since $\theta_0\ne 0$, Theorem~\ref{thm:cohen} implies that there exists a primitive element $\alpha\in \F_{q^n}$ such that $\Tr_{n/\frac{n}{t}}(\alpha)=\theta_0$ and then, by the transitivity of the trace, we have that
$$\Tr_{n/d_i}(\alpha)=\Tr_{\frac{n}{t}/d_i}(\theta_0)=a_i, \; 1\le i\le k.$$

\section{Theorem \ref{thm:main} under the condition $\gcd(d_i, d_j)=1$}\label{sec:concasymp}

In this section, we discuss the existence of primitive elements of $\F_{q^n}$ with arbitrary prescribed traces over extensions $\F_{q^{d_i}}$ under the following condition:
\begin{center}
\fbox{$\gcd(d_i, d_j)=1\;\,\text{for}\;\,1\le i<j\le k$}.
\end{center}
We observe that the former condition is not restrictive when $k=2$. In fact, if $d_1<d_2$ are divisors of $n$ and $d=\gcd(d_1, d_2)$, we have that $\F_{q^{d_i}}=\F_{Q^{e_i}}$ where $e_i=d_i/d$ satisfy $\gcd(e_1, e_2)=1$. We obtain asymptotic and concrete results, that are displayed in the following theorem.

\begin{thm}\label{thm:cop}
Let $n>1$ be an integer that is not a prime power, $1 < k < \sigma_0(n)$ and $\mathbf{d} = (d_1,\dots,d_k) \in \Lambda_k(n)$ be such that $\gcd(d_i, d_j)=1$ for every $1\le i<j\le k$. Furthermore, let $\mathbf{a} = (a_1,\dots,a_k) \in \F(\mathbf{d})$ be a $\mathbf{d}$-admissible $k$-tuple. 
Then there exists a primitive element $\alpha \in \fqn$ with prescribed traces $\Tr_{n/d_i}(\alpha) = a_i$ provided that one of the following holds:
\begin{enumerate}[(a)]
\item $k\ge 3$;
\item $k=2$ and 
\begin{enumerate}[{(b.}1)]
\item $d_1\ge 5$ and $q\ge 5$ if $(d_1,d_2)=(5,6)$; 
\item $d_1 = 4$ and either $d_2 \ge 11$, or $d_2 \ge 9$ and $q\ge 3$, or $d_2=5, 7$ and $q\ge e^{e^{6.7}}$;
\item $d_1 = 3$ and either $d_2 \ge 38$, or $d_2 \ge 5$ and $q\ge e^{e^{26.1}}$.
\end{enumerate}
\end{enumerate}
\end{thm}

\begin{proof}
We may assume that $\lcm(d_1, \ldots, d_k)=n$ since otherwise the result is directly implied by Theorem~\ref{thm:main-2}. For $k\ge 2$, the condition $\gcd(d_i, d_j)=1$ implies that
$$\lambda(\mathbf{d}) = d_1 + \dots + d_k - k + 1.$$
From Lemma~\ref{lem:tech-2}, Theorem~\ref{thm:main} and the equation above, it suffices to verify that
\begin{equation}\label{enough1}
\frac{n}{2} - \sum_{i=1}^kd_i + k - 1 \ge \frac{n}{a} + \log_q(c_{q^n,a}),
\end{equation}
for some $a \ge 3$. We always take $a=4$ or $a=8$. Since $\lcm(d_1, \ldots, d_k)=n$, $2 \le d_1 < \dots < d_k$ and $\gcd(d_i,d_j)=1$, we have that 
\begin{equation}
\label{eq:prod-n}d_1\dots d_k = n.
\end{equation}
We provide the proof of items (a) and (b) separately.

\subsection{The case $k\ge 3$}

We split the proof into cases.
\begin{enumerate}[(i)] 
\item $k \ge 4$: Let $2=p_1<p_2<\cdots$ be the increasing sequence of the prime numbers. We have that $p_{\ell}\le d_{\ell}$ and then
 $$p_{\ell} \le d_{\ell} \le \left(\frac{n}{p_1\dots p_{\ell-1}}\right)^{\frac{1}{k+1-\ell}},$$
for $1 \le \ell \le k$, where the empty product equals $1$. Furthermore,
\begin{equation}\label{ineqn}
\frac n2 - \left[ \log_7{\left(\frac{n}{30}\right)} + 2 \right]\sqrt{\frac n2} - \frac n{30} + 3 \ge \frac n4 + 2
\end{equation}
for $n \ge 39$. Since $\left({\frac{n}{p_1\dots p_{\ell-1}}}\right)^{\frac{1}{k+1-\ell}} < \sqrt{\frac n2}$ for every $1 \le \ell \le k-1$ and $n \ge 2 \cdot 3 \cdot 5 \cdot 7^{k-3} \ge 210$, the LHS of inequality \eqref{enough1} is greater than the LHS of inequality \eqref{ineqn}. Taking $a=4$, the RHS of inequality \eqref{ineqn} is greater than the RHS of inequality \eqref{enough1}. This concludes the case $k\ge 4$.

\item $k = 3$: In the same way, we have that the following inequality
\begin{equation}\label{ineq3}
\frac n2 - \sqrt[3]{n} - \sqrt{\frac{n}{2}} - \frac n{6} + 2 \ge \frac n8 + 12.2
\end{equation}
holds true for $n \ge 107$. Since the LHS of inequality \eqref{enough1} is greater than the LHS of inequality \eqref{ineq3}, and the RHS of inequality \eqref{ineq3} is greater than the RHS of inequality \eqref{enough1} with $a=8$, we are done unless $n \in \{30=2\cdot3\cdot5, \; 42=2\cdot3\cdot7, \; 60=3\cdot4\cdot5, \; 66=2\cdot3\cdot11, \; 70=2\cdot5\cdot7, \; 78=2\cdot3\cdot13, \; 84=3\cdot4\cdot7, \; 90=2\cdot5\cdot9, \; 102=2\cdot3\cdot17, \; 105=3\cdot5\cdot7\}$, which are the numbers smaller than $107$ that split into at least three non-trivial relatively prime factors. We now consider the cases:
\begin{enumerate}[{(ii.}1)]
\item $(d_1,d_2) = (3,5)$: If $n \ge 60$ then $\frac n2 - 3 - 5 - \frac n{15} + 2 \ge \frac n4 + 2$, and we argue as above with $a=4$.
\item $(d_1,d_2) = (3,4)$: If $n \ge 42$ then $\frac n2 - 3 - 4 - \frac n{12} + 2 \ge \frac n4 + 2$, and we argue as above with $a=4$.
\item $(d_1,d_2) = (2,5)$: If $n \ge 47$ then $\frac n2 - 2 - 5 - \frac n{10} + 2 \ge \frac n4 + 2$, and we argue as above with $a=4$.
\item $(d_1,d_2) = (2,3)$: If $n \ge 60$ then $\frac n2 - 2 - 3 - \frac n6 + 2 \ge \frac n4 + 2$, and we argue as above with $a=4$. Hence, there only remains $n \in \{30,42\}$.

For $n=30$, the inequality $$\frac{30}2 - 2 - 3 - 5 + 2 \ge \frac{30}8 + \log_q(4514.7)$$ holds true if $q \ge 14$ and we argue as above with $a = 8$. Therefore there only remains the cases $q \in \{2,3,4,5,7,8,9,11,13\}$, for which the inequality $q^{n/2 - \lambda(\mathbf{d})} \ge W(q^n-1)$ can be directly verified.

For $n=42$, the inequality $$\frac{42}2 - 2 - 3 - 7 + 2 \ge \frac{42}8 + \log_q(4514.7)$$ holds true  if $q \ge 5$ and we argue as above with $a = 8$. Therefore it only remains the cases $q \in \{2,3,4\}$, for which the inequality $q^{n/2 - \lambda(\mathbf{d})} \ge W(q^n-1)$ can be directly verified.
\end{enumerate}
\end{enumerate}

\subsection{The case $k=2$}

As in previous cases, it suffices to prove that 
\begin{equation}\label{ineq2}
\frac n2 - d_1 - d_2 + 1 \ge \frac na + \log_q(c_{q^n,a}),
\end{equation}
for $a \in \{4,8\}$. Notice that if $d_1 \ge 8$ then $d_2 \ge 9$ and $(d_1 - 4)(d_2 - 4) \ge 20 \ge 12 + 4 \log_q(c_{q^n,4})$. Therefore, inequality \eqref{ineq2} holds true with $a=4$. Table~\ref{tab:1} provides the ranges of $q, d_1, d_2$ where inequality \eqref{ineq2} holds and the value of $a$ that is used.
{\footnotesize
\begin{center}
\begin{table}[h]
\begin{multicols}{2}
\qquad
\begin{tabular}{|c|c|c|c|}
\hline
$d_1$ & $d_2$ & $q$ & $a$ \\
\hline
\hline
$7$ & $\ge 11$ & $\forall \; q$ & $4$ \\
\hline
$7$ & $10$ & $\ge 5$ & $4$ \\
\hline
$7$ & $9$ & $\ge 8$ & $4$ \\
\hline
$7$ & $8$ & $\ge 37$ & $8$ \\
\hline
$6$ & $\ge 15$ & $\forall \; q$ & $4$ \\
\hline
$6$ & $13$ & $\ge 3$ & $4$ \\
\hline
$6$ & $11$ & $\ge 3$ & $8$ \\
\hline
$6$ & $7$ & $\ge 11$ & $8$ \\
\hline
$5$ & $\ge 19$ & $\forall \; q$ & $8$ \\
\hline
$5$ & $\ge 14$ & $\ge 3$ & $8$ \\
\hline
$5$ & $\ge 12$ & $\ge 4$ & $8$ \\
\hline
$5$ & $11$ & $\ge 5$ & $8$ \\
\hline
%
$5$ & $9$ & $\ge 9$ & $8$ \\
\hline
$5$ & $8$ & $\ge 17$ & $8$ \\
\hline
$5$ & $7$ & $\ge 53$ & $8$ \\
\hline
$5$ & $6$ & $\ge 839$ & $8$ \\
\hline
$4$ & $\ge 31$ & $\forall q$ & $8$ \\
\hline
$4$ & $\ge 23$ & $\ge 3$ & $8$ \\
\hline
\end{tabular}

\begin{tabular}{|c|c|c|c|}
\hline
$d_1$ & $d_2$ & $q$ & $a$ \\
\hline
\hline
$4$ & $\ge 19$ & $\ge 4$ & $8$ \\
\hline
$4$ & $17$ & $\ge 5$ & $8$ \\
\hline
$4$ & $15$ & $\ge 7$ & $8$ \\
\hline
$4$ & $13$ & $\ge 13$ & $8$ \\
\hline
$4$ & $11$ & $\ge 29$ & $8$ \\
\hline
$4$ & $9$ & $\ge 274$ & $8$ \\
\hline
%
$4$ & $7$ & $\ge 2.039 \cdot 10^7$ & $8$ \\
\hline
$3$ & $\ge 114$ & $\forall q$ & $8$ \\
\hline
$3$ & $\ge 78$ & $\ge 3$ & $8$ \\
\hline
$3$ & $\ge 65$ & $\ge 4$ & $8$ \\
\hline
$3$ & $\ge 58$ & $\ge 5$ & $8$ \\
\hline
$3$ & $\ge 52$ & $\ge 7$ & $8$ \\
\hline
$3$ & $\ge 49$ & $\ge 8$ & $8$ \\
\hline
$3$ & $47$ & $\ge 9$ & $8$ \\
\hline
$3$ & $46$ & $\ge 11$ & $8$ \\
\hline
$3$ & $\ge 43$ & $\ge 13$ & $8$ \\
\hline
$3$ & $\ge 40$ & $\ge 17$ & $8$ \\ 
\hline
$3$ & $38$ & $\ge 23$ & $8$ \\ 
\hline
\end{tabular}
\end{multicols}
\caption{Ranges of $q, d_1$ and $d_2$ where inequality \eqref{ineq2} holds true and the value of $a$ that is used.}\label{tab:1}
\end{table}
\end{center}}

Table~\ref{tab:1} compiles exceptions $(q, d_1, d_2)$ for inequality \eqref{ineq2}, assuming that either $d_1\ge 5$, or $d_1=4$ and $d_2\ge 7$, or $d_1=3$ and $d_2\ge 38$. With respect to these ranges, the exceptional triples have reasonably small parameters. Using the software SageMath we verify that such triples $(q, d_1, d_2)$ satisfy inequality \eqref{ineqprincipal}, with the exception of $(q, d_1,d_2) = (q, 5,6)$ with $q < 5$ and $(q, d_1, d_2) = (2, 4, 9)$.

For $(q,d_1,d_2) = (q,3,4)$, inequality \eqref{ineqprincipal} does not hold for any prime power $q$. For the remaining cases, that is, $(q,d_1,d_2) = (q,4,5)$ and $(q,d_1,d_2) = (q,3,d_2)$ with $5 \le d_2 \le 37$, Lemma \ref{lem:tech-1}(i) (or inequality \eqref{ineq2} with $a=8$ provided $d_1=3$ and $17 \le d_2 \le 37$) ensures that there exists a computable constant $q_0$ depending only on $d_1$ and $d_2$ such that inequality \eqref{ineqprincipal} holds for every $q \ge q_0$.

\end{proof}

\begin{rem}\label{rmk:k=d1=2}
For $k = 2$, $d_1 = 2$ and $n = \lcm(2,d_2)$, we have that $q^{n/2 - \lambda(\mathbf{d})} < 1 < W(q^n - 1)$, and so Theorem~\ref{thm:main} is inconclusive. Moreover, in this setting we can actually provide genuine exceptions. In fact, let $n=2\cdot N$, $N>1$ odd and choose $b\in \F_{q^2}$ in  way that $\Tr_{N/1}(b)=0$. In particular, the pair $(b, 0)\in \F_{q^2}\times \F_{q^N}$ is $(2, N)$-admissible. However, there is no primitive element $\alpha\in \F_{q^n}$ with zero trace over $\F_{q^N}$. In fact, any element $\alpha\in \F_{q^n}$ with zero trace over $\F_{q^N}$ satisfies $\alpha^{q^N}=-\alpha$ and, consequently, $\alpha^{2(q^N-1)}=1$. But such an $\alpha$ cannot be primitive since $2(q^N-1)<q^{n}-1$ for every $q\ge 2$.
\end{rem}

\section{Conclusion}
In this paper we have discussed the existence of primitive elements of finite fields with prescribed traces in intermediate extensions. Our main result provides a sufficient condition for the existence of such elements. This condition is encoded in an inequality, which is further explored in order to obtain concrete results on the existence of these elements; this is presented in Theorem~\ref{thm:cop}. 

It would be desirable to explore the validity of Theorem~\ref{thm:cop} without the restriction $\gcd(d_i, d_j)=1$ or at least complete this theorem, exploring the remaining cases under this restriction. For instance, by using a sieving method that is traditional in this kind of problem (see~\cite{cohen, cohen-trud}), one can remove the restrictions $q\ge 5$ and $q \ge 3$ in items (b.1) and  (b.2) of Theorem~\ref{thm:cop}. Within the approach of this paper, we believe that any such improvement would have to go through sharper estimates on the character sums appearing in the proof of Theorem~\ref{thm:main}.

\section*{Acknowledgement}

We thank the anonymous referee for the suggestions that substantially improved the presentation of this work.


\begin{thebibliography}{0}
\bibitem{Cohentrace}
S.~D. Cohen.
\newblock Primitive elements and polynomials with arbitrary trace.
\newblock {\em Discrete Math.}, 83(1):1--7, 1990.

\bibitem{cohen}
S.~D. Cohen and S.~Huczynska.
\newblock The primitive normal basis theorem {{--}} without a computer.
\newblock {\em J. London Math. Soc.}, 67(1):41--56, 2003.

\bibitem{cohen-trud} 
S.~D.~Cohen., T.~Oliveira e Silva and T.~Trudgian.
\newblock On consecutive primitive elements in a finite field. 
\newblock {\em Bull. Lond. Math. Soc.} 47(3): 418--426, 2015.

\bibitem{dh}
W.~Diffie and M.~ Hellman.
\newblock New directions in cryptography.
\newblock {\em IEEE Trans. Information Theory}, 22(6):644--654, 1976.

\bibitem{KR}
G.~Kapetanakis and L.~Reis.
\newblock Variations of the Primitive Normal Basis Theorem.
\newblock {\em Des. Codes Cryptogr.} 87(7): 1459--1480, 2019.

\bibitem{lenstra}
H.~W. Lenstra, Jr and R.~J. Schoof.
\newblock Primitive normal bases for finite fields.
\newblock {\em Math. Comp.}, 48(177):217--231, 1987.

\bibitem{LN}
R.~Lidl and H.~Niederreiter.
\newblock {\em Finite Fields}, volume~20 of {\em Encyclopedia of Mathematics and
  its Applications}.
\newblock Cambridge University Press, Cambridge, second edition, 1997.

\bibitem{LR19}
L.~Reis.
\newblock Counting solutions of special linear equations over finite fields.
\newblock {\em Finite Fields Appl.} 68: 101759, 2020.
\end{thebibliography}

\end{document}